\documentclass[a4paper,twoside,11pt]{article}
\usepackage{amssymb}
\usepackage{amsmath}
\usepackage{graphicx}
\usepackage{amsthm}
\usepackage{array}

\usepackage{amsmath,latexsym,amssymb,amsfonts,amsbsy,amsthm,mathrsfs}
\usepackage{algorithm}
\usepackage{algorithmic}
\usepackage[top=1in, bottom=1in, left=1.25in, right=1.25in]{geometry}
\usepackage{epstopdf}
\usepackage{diagbox}
\usepackage{enumerate}
\usepackage{color}
\usepackage{indentfirst}
\usepackage{braket}
\usepackage{booktabs}
\usepackage{multicol}
\usepackage{multirow}
\usepackage{float}
\usepackage{subfig, graphicx}

\usepackage{hyperref}
\hypersetup{
colorlinks=true,
citecolor=blue,
linkcolor=blue,
urlcolor=blue
}
\graphicspath{{figures/}}
\newfont{\bb}{msbm10}

\newtheorem{theorem}{Theorem}[section]

\newtheorem{lemma}{Lemma}[section]

\newtheorem{remark}{Remark}[section]

\numberwithin{equation}{section}

\title{OMEGA THEORMES FOR LOGARITHMIC DERIVATIVES OF ZETA AND $L$-FUNCTIONS NEAR THE 1-LINE}

\author{
Zhonghua Li\\
School of Mathematical Sciences \\
Key Laboratory of Intelligent Computing and Applications(Ministry of Education) \\
Tongji University, Shanghai 200092, CHINA\\
Email: zhonghua-li@tongji.edu.cn\\
and \\
Shengbo Zhao\footnote{Corresponding Author}\\
School of Mathematical Sciences \\
Key Laboratory of Intelligent Computing and Applications(Ministry of Education) \\
Tongji University, Shanghai 200092, CHINA\\
Email: shengbozhao@hotmail.com\\
}

\date{}
\begin{document}
\cleardoublepage \pagestyle{myheadings}
\maketitle

\begin{abstract}
	We establish an omega theorem for logarithmic derivative of the Riemann zeta function near the 1-line by resonance method. We show that the inequality $\left| \zeta^{\prime}\left(\sigma_A+it\right)/\zeta\left(\sigma_A+it\right) \right| \geqslant \left(\left(e^A-1\right)/A\right)\log_2 T + O\left(\log_2 T / \log_3 T\right)$ has a solution $t \in [T^{\beta}, T]$ for all sufficiently large $T,$ where $\sigma_A = 1 - A / \log_2 {T}.$
	Furthermore, we give a conditional lower bound for the measure of the set of $t$ for which the logarithmic derivative of the Riemann zeta function is large.
	Moreover, similar results can be generalized to Dirichlet $L$-functions.
\end{abstract}


\bigskip

\section{Introduction}
\label{Introduction}
Let $\zeta \left(s \right)$ be the Riemann zeta function and $s = \sigma + it,$ where $\sigma$ and $t$ are real numbers. The logarithmic derivative of the Riemann zeta function is one of the significant objects in the analytic number theory, 
as it plays a key role in the proof of the prime number theorem. 
The study of it can be traced back to the 1910s. In 1911, Landau \cite{landau1911theorie} proved that 
there exists a positive constant $k_1$ such that $\left| \zeta^{\prime}\left(s\right)/\zeta\left(s\right) \right| \geqslant k_1 \log_2 t$ has a solution $s= \sigma + i t $ in $\{\sigma+it \operatorname{:} \sigma>1, t> \tau \}$ for any given $\tau >0,$ where $\log_2 {t} \operatorname{:}= \log{\log{t}}. $
In 1913, this result had been improved by Bohr and Landau \cite{bohr1913beitrage}. They proved that 
there exists a positive constant $k_2$ such that $\left| \zeta^{\prime}\left(s\right)/\zeta\left(s\right) \right| \geqslant \left(\log {t}\right)^{k_2 \theta } $ has a solution $s= \sigma + i t $ in  $\{\sigma+it \operatorname{:} \sigma \geqslant 1- \theta , t \geqslant \tau \}$ for any given $\theta \in \left(0,\frac{1}{2}\right)$ and for any given $\tau >0.$
\par
Assuming the Riemann hypothesis(RH), Littlewood \cite{littlewood1926riemann} showed the classical upper bound
$$\left| \frac{\zeta^{\prime}}{\zeta}\left(\sigma+it\right)\right| \leqslant \sum_{n \leqslant \left(\log{t}\right)^ 2} \frac{\Lambda \left(n\right)}{n^{\sigma}}+ O\left(\left(\log{t}\right)^{2-2\sigma}\right)$$
holds uniformly for $\frac{1}{2}+ 1/\log_2 t \leqslant \sigma \leqslant \frac{3}{2}$ and $ t \geqslant 10.$ Until now, this bound has not been improved significantly, especially its main terms. 
When assuming RH, Chirre, Val$\rm\mathring{a}$s Hagen and Simoni$\rm \check{c}$ \cite{chirre2024conditional} gave the following conditional upper bound on the 1-line
$$\left| \frac{\zeta^{\prime}}{\zeta}\left(1+it\right) \right| \leqslant 2\log_2 t - 0.4989 + 5.35\frac{\left(\log_2 t\right)^2}{\log{t}}, ~ \forall t \geqslant 10^{30}.$$
\par
In 2006, Ihara \cite{ihara2006euler} focused on the study of the Euler-Kronecker constants of the cyclotomic fields.
Values of $L^{\prime}\left(1, \chi\right)/L\left(1, \chi\right)$ appear naturally in this study, and there is a certain correlation between them.
This provides new motivation to study the logarithmic derivative of Dirichlet $L$-functions.
\par
Assuming the Generalized Riemann hypothesis(GRH), Ihara, Kumar Murty and Mahoro Shimura \cite{IharaMurtyShimura2009} obtained the following conditional upper bound
$$\left| \frac{L^{\prime}}{L}\left(1,\chi\right) \right| \leqslant 2\log_2 q + 2\left(1-\log{2}\right) + O\left(\frac{\log_2 q}{\log{q}}\right)$$
for all non-principal primitive characters $\chi \left(\operatorname{mod}~q\right).$ Later, their result has been optimized greatly.
For instance, assuming GRH, Chirre, Val$\rm\mathring{a}$s Hagen and Simoni$\rm \check{c}$ \cite{chirre2024conditional} established the conditional upper bound as follows
$$\left| \frac{L^{\prime}}{L}\left(1,\chi\right) \right| \leqslant 2\log_2 q - 0.4989 +5.91\frac{\left(\log_2 q\right)^2}{\log{q}}, ~ \forall q \geqslant 10^{30},$$
where $\chi $ is a primitive character modulo $q.$ Note that $-0.4989 < 2\left(1 - \log{2}\right) = 0.6137 \cdots .$
In fact, in the past years, there were many results on upper bounds of logarithmic derivatives of Dirichlet $L$-functions. 
For instance, we can see \cite{PalojärviSimonič2022,Simonič2023,Trudgian2015}.
\par
There are much fewer omega results, in other words, lower bounds of maximal values of logarithmic derivatives of $\zeta\left(s\right)$ and $L\left(s, \chi\right).$
Mourtada and Kumar Murty \cite{mourtada2013omega} proved that there are inﬁnitely many fundamental discriminants $D,$ regardless of whether they are positive or negative, such that
$L^{\prime}\left(1, \chi_D\right)/L\left(1,\chi_D\right) \geqslant \log_2 \left|D\right| +O\left(1\right),$ where $\chi_D$ is the quadratic Dirichlet character of conductor $D.$
The key point in proving this result is making use of the explicit formula in \cite{ihara2009logarithmic}. Furthermore, assuming GRH, they also proved that for $x$ large enough, 
there are $\gg x^{\frac{1}{2}}$ primes $q \leqslant x$ such that $- L^{\prime}\left(1, \chi_q\right)/L\left(1,\chi_q\right) \geqslant \log_2 x + \log_3 x +O\left(1\right),$ where $\log_3 x \operatorname{:}= \log {\log {\log {x}}}.$
\par
In 2023, Paul \cite{paul2023values} gave that if $f \in S_k \left(\operatorname{SL_2 \left(\mathbb{Z} \right)}\right)$ is a normalized Hecke eigenform, 
then there are inﬁnitely many fundamental discriminants $D$ and quadratic Dirichlet characters $\chi \left(\operatorname{mod}~ D\right)$ such that
$L^{\prime}/L \left(1, f, \chi\right) \geqslant \frac{1}{8} \log_2 \left(Dk\right) + O\left(1\right),$
where we denote $L\left(s, f, \chi\right)$ as Dirichlet $L$-functions associated by $f$ and $\chi \left(\operatorname{mod}~ D\right).$
\par
Yang \cite{yang2023omega} established several omega theorems for any $\sigma \in (\frac{1}{2}, 1],$ not just limited to the case $\sigma = 1.$ Fix $\beta \in \left(0,1\right),$ he showed that
$$\max_{T^{\beta} \leqslant t \leqslant T} - \operatorname{Re} \frac{\zeta^{\prime}}{\zeta} \left(1+it\right) \geqslant \log_2 T +\log_3 T + C_1 - \epsilon,$$
where $\epsilon \in \left(0,1 \right)$ and $C_1$ is an explicit constant.
In addition, when $\sigma \in \left(\frac{1}{2}, 1\right),$ he showed that
$$\max_{T^{\beta} \leqslant t \leqslant T} - \operatorname{Re} \frac{\zeta^{\prime}}{\zeta} \left(\sigma+it\right) \geqslant C_3\left(\sigma\right) \left(\log{T}\right)^{1-\sigma}\left(\log_2 T\right)^{1-\sigma},$$
where $C_3 \left(\sigma\right)$ is a constant. Similar results on Dirichlet $L$-functions had been given as well.
\par
We take inspiration from Yang's exceptional work and harvest some similar results. We focus on omega theorems for logarithmic derivatives of $\zeta \left(s\right)$ and $L\left(s, \chi \right)$ near the 1-line. 
Specifically, we set $\sigma_A = 1- A/ \log_2 T$ for all sufficiently large $T$ when we concentrate on the Riemann zeta function. Naturally, we set $\sigma_A = 1- A/ \log_2 q$ as $q \to \infty$ when we consider Dirichlet $L$-functions.
To be honest, we can shift the proof of the theorems about $\zeta\left(s\right)$ to the conclusions on $L\left(s, \chi\right)$, as the methods of getting them are roughly same.
\par
We have the following result for the logarithmic derivative of the Riemann zeta function near the 1-line.
\begin{theorem}
	\label{thm1}
	Fix $\beta \in \left(0,1\right).$ Let $A$ be any positive real number, $\sigma_{A} \operatorname{:} = 1-A/\log_2 T.$ Then for all sufficiently large $T$, we have 
	$$\max_{T^{\beta} \leqslant t \leqslant T} - \operatorname{Re} \frac{\zeta ^{\prime} }{\zeta }\left( \sigma_{A} + it\right) \geqslant \frac{e^{A} -1}{A} \log_2 T + O\left(\frac{\log_2 {T}}{\log_3 {T}}\right) .$$
    In particular, for all sufficiently large $T,$ we have
	$$\max_{T^{\beta} \leqslant t \leqslant T} \left| \frac{\zeta ^{\prime} }{\zeta }\left( \sigma_{A} + it\right) \right| \geqslant \frac{e^{A} -1}{A} \log_2 T + O\left(\frac{\log_2 {T}}{\log_3 {T}}\right) .$$
\end{theorem}
\par
With the help of the above theorem and methods of \cite{soundararajan2008extreme, yang2023omega}, we obtain the next theorem, which gives a conditional lower bound for the 
measure of the set $\mathcal{M} \left(T,x\right).$ By observing this set, we can show the distribution of values of logarithmic derivatives of the Riemann zeta function.
\begin{theorem}
	\label{thm2}
	Assuming RH. Fix $\beta \in \left(0,1 \right).$ Let $x > 0 $ be given. Define the set $\mathcal{M} \left(T,x\right)$ as 
	$$\mathcal{M} \left(T,x\right) = \bigg\{t\in [T^{\beta}, T] \operatorname{:} - \operatorname{Re} \frac{\zeta^{\prime}}{\zeta}\left(\sigma_A + it \right) \geqslant \frac{\log_2 T}{\log_3 T}\bigg(\frac{e^A-1}{A}\log_3 T - x\bigg) \bigg\}.$$
	Then we have 
	$$\operatorname{meas}\left(\mathcal{M} \left(T,x\right)\right) \geqslant T^{1-\left(1-\beta\right)e^{-x} + o\left(1\right)},~ as ~~ T \to \infty.$$
\end{theorem}
\par
Generalize the above results to Dirichlet $L$-functions, we obtain two similar theorems as follows.
\begin{theorem}
	\label{thm3}
	Let $A$ be any positive real number, $\sigma_{A} \operatorname{:} = 1-A/\log_2 q.$ Then for all sufficiently large prime $q$, we have 
	$$\max_{\substack{\chi \neq \chi_{0} \\ \chi \left(\operatorname{mod}~ q\right)}} - \operatorname{Re} \frac{L^{\prime}}{L}\left(\sigma_{A},\chi \right) \geqslant \frac{e^{A} -1}{A} \log_2 q + O\left(\frac{\log_2 {q}}{\log_3 {q}}\right),$$
    where $\chi_0$ is the principal character modulo $q$, and the above maximal is taken over all non-principal characters $\chi$ which modulo $q.$
	In particular, for all sufficiently large prime $q,$ we have
	$$\max_{\substack{\chi \neq \chi_{0} \\ \chi \left(\operatorname{mod}~ q\right)}} \left| \frac{L^{\prime}}{L}\left(\sigma_{A},\chi \right)\right| \geqslant \frac{e^{A} -1}{A} \log_2 q + O\left(\frac{\log_2 {q}}{\log_3 {q}}\right).$$
\end{theorem}
\begin{theorem}
	\label{thm4}
	Assuming GRH. Let $x > 0 $ be given. Define the set $\mathcal{N} \left(q,x\right)$ as 
	$$\mathcal{N} \left(q,x\right) = \bigg\{\chi \left(\operatorname{mod}~q\right) \operatorname{:} - \operatorname{Re} \frac{L^{\prime}}{L}\left(\sigma_A , \chi\right) \geqslant \frac{\log_2 {q}}{\log_3 {q}}\left(\frac{e^{A} -1}{A} \log_3 q - x\right)\bigg\}.$$
	Then for prime $q \to \infty,$ we have
	$$\#\mathcal{N} \left(q,x\right) \geqslant q^{1-e^{-x}+ o \left(1\right)}.$$
\end{theorem}
\par
We will make use of resonance methods to obtain our theorems in this paper. 
In 1988, the resonance method was first invented by Voronin \cite{voronin1988lower}. 
This method is an effective and ingenious way to get the lower bounds of the complex functions by taking the help of some ``resonators''. 
But unfortunately, his work was not taken seriously at that time.
About twenty years later, Soundararajan \cite{soundararajan2008extreme} introduced a broader and more practical method based on Voronin's work, which called short resonator methods. 
Later, Aistleitner \cite{aistleitner2016lower} gave another long resonator methods, developing the resonance methods. 
We recommend readers to see \cite{Aistleitner2019IMRN, Aistleitner2019QJMath, Bondarenko2018MathAnn, soundararajan2008extreme, voronin1988lower} to get more information on resonance methods.
\par 
In this paper, we will employ long resonator methods. The main idea is to find a Dirichlet series $R\left(t\right) = \sum_{n \in \mathbb{N}} r\left(n\right) n ^{- i t},$ named resonator, 
to resonate with logarithmic derivatives of the Riemann zeta function and pick out its large values.
\par
The outline of this paper is as follows. In Section \ref{preliminarylemmas} we present some preliminary lemmas. In Section \ref{proofthm1} we prove Theorem \ref{thm1} by the resonance method. In Section \ref{proofthm2} we prove Theorem \ref{thm2}. In fact, the proof is a slightly modification of the proof of Theorem \ref{thm1}. The proof of Theorems \ref{thm3} and \ref{thm4} is provided in Sections \ref{proofthm3} and \ref{proofthm4}, respectively. 
\par
~\\
\indent \textbf{Notation.}~In this paper, we denote $p$ and $q$ as prime numbers. We denote $\chi$ as Dirichlet characters and $\chi_0$ as the principal character. 
We denote $\Lambda \left(n\right)$ as the Mangoldt function, and $\vartheta \left(x\right) = \sum_{p \leqslant x}\log{p}$ as the Chebyshev's function.
In addition, we use the short-hand notations, $\log_2 T \operatorname{:} = \log{\log{T}},~ \log_3 T \operatorname{:} = \log{\log{\log T}}$ and so on. 

\section{Preliminary Lemmas}
\label{preliminarylemmas}
First, we need an approximate formula for $\zeta^{\prime}\left( \sigma_{A} + it\right)/ \zeta \left(\sigma_{A} + it\right)$. We establish the following lemma, which is useful when combining with zero-density result of $\zeta \left(s\right).$ 
This lemma is similar to \cite[Lemma 1]{yang2023omega}.
 \begin{lemma}
	\label{approximatezeta}
    Let $A$ be any positive real number, $T$ be sufficiently large real number. Set
    $\sigma_{A} \operatorname{:} = 1-A/\log_2 T.$
    Let $Y \geqslant 3, t \geqslant Y+3$ and $\frac{1}{2} \leqslant \sigma_0 \leqslant 1 .$ Suppose that the rectangle $\{s:\sigma_0 <\operatorname{Re}\left(s\right) \leqslant 1,\lvert \operatorname{Im}\left(s\right) - t \rvert \leqslant Y+2 \}$ 
	is free of zeros of $\zeta \left(s\right).$ Then for $\sigma_{A} \left(\leqslant 3\right)$ and any $\xi \in \left[t-Y , t+Y\right],$ we have
	$$\left|\frac{\zeta^{\prime}}{\zeta}\left(\sigma_{A} + i\xi \right) \right| \ll \frac{\log t}{\sigma_{A}-\sigma_{0}}.$$
	\par
	Further, for $\sigma_{1} \in \left(\sigma_{0},\sigma_{A}\right),$ we have 
	$$-\frac{\zeta ^{\prime}}{\zeta}\left(\sigma_{A} + it \right)= \sum_{n \leqslant Y}\frac{\Lambda \left(n\right)}{n^{\sigma_{A} + it}} + O\left(\frac{\log{t}}{\sigma_{1}-\sigma_{0}}Y^{\sigma_{1}-\sigma_{A}} \log {\frac{Y}{\sigma_{A}-\sigma_{1}}}\right).$$
\end{lemma}
\begin{proof}
	We write $\rho = \alpha +i \gamma, 0 < \alpha <1 $ as the non-trivial zeros of $\zeta(s).$ By  \cite[Lemma 8.2]{Koukoulopoulos2019}, we have 
$$\frac{\zeta^{\prime}}{\zeta}\left(s \right) = - \frac{1}{s-1}+\sum_{\left| \gamma -t \right| \leqslant 1} \frac{1}{s- \rho } + O\left(\log{\left|s\right|}+2\right).$$
Then we have 
$$\left|\frac{\zeta^{\prime}}{\zeta}\left(\sigma_{A} + i\xi \right) \right| \leqslant \sum_{\left|\gamma -\xi \right| \leqslant 1} \frac{1}{\sigma_{A}-\alpha } + O\left(\log{\xi }\right) \ll \frac{\log{t}}{\sigma_{A}-\sigma_{0}}$$
because of our assumption on the rectangle. Now we obtain the first claim.
\par
Taking $c=-\sigma_{A}+1+1/\log{Y}$ in Perron's formula of \cite{Koukoulopoulos2019}, we find the following asymptotic formula
\begin{equation}
	\label{Perron}
	\sum_{n \leqslant Y}\frac{\Lambda \left(n\right)}{n^{\sigma_{A} + it}}=\frac{1}{2\pi i}\int_{c-iY}^{c+iY} - \frac{\zeta^{\prime}\left(\sigma_{A}+it+s\right)}{\zeta \left(\sigma_{A}+it+s\right)} \frac{Y^{s}}{s} \mathrm{d}s + O\left(Y^{-\sigma_{A}}\left(\log{Y}\right)^2\right).
\end{equation}
\par 
Note that if $\sigma_1 - \sigma_{A} \leqslant \operatorname{Re}\left(s\right) \leqslant c$ and $\left| \operatorname{Im}\left(s\right) \right| \leqslant Y, $ then 
$\sigma_1 \leqslant \operatorname{Re}\left(\sigma_A +it +s\right) \leqslant 1+1/\log{Y}$ and $\left| \operatorname{Im}\left(\sigma_A +it +s\right) - t\right| \leqslant Y.$
As the rectangle is a zero-free region of $\zeta \left(s\right),$ by the residue theorem, we can move the line of integration from the line $\operatorname{Re}\left(s\right) = c$ to the line $\operatorname{Re}\left(s\right) = \sigma_1 - \sigma_A$ to obtain
	\begin{align}
		\frac{1}{2\pi i}\int_{c-iY}^{c+iY} - \frac{\zeta^{\prime}\left(\sigma_{A}+it+s\right)}{\zeta \left(\sigma_{A}+it+s\right)} &\frac{Y^{s}}{s} \mathrm{d}s  = - \frac{\zeta^{\prime}\left(\sigma_A +it\right)}{\zeta\left(\sigma_A +it\right)} \nonumber \\
	 &+ \frac{1}{2\pi i}\left(\int_{c-iY}^{\sigma_1-\sigma_A-iY} +\int_{\sigma_1-\sigma_A-iY}^{\sigma_1-\sigma_A+iY} + \int_{\sigma_1-\sigma_A+iY}^{c+iY}\right). \label{residue}
	\end{align}
For simplicity, we omitted integrands on the right-hand side of \eqref{residue} without causing confusion.
\par
According to the first claim of this lemma, the first and third integrals on the right-hand side of \eqref{residue} are bounded by
\begin{equation}
	\label{tom}
	\ll \int_{\sigma_1-\sigma_A}^{1-\sigma_{A} + \frac{1}{\log{Y}}}\frac{\log{t}}{\sigma_A + x - \sigma_0} \frac{Y^x}{Y} \mathrm{d}x \ll \frac{\log{t}}{\left(\sigma_1- \sigma_0\right)\log{Y}}Y^{-\sigma_A}, 
\end{equation}meanwhile, the second integral on the right-hand of \eqref{residue} is bounded by 
\begin{equation}
	\label{jerry}
	\ll \int_{-Y}^{Y} \frac{\log{t}}{\sigma_1 - \sigma_0} \frac{Y^{\sigma_1 - \sigma_A }}{\sqrt{\left(\sigma_1 - \sigma_A \right)^2 + \xi ^2}} \mathrm{d}\xi 
	\ll \frac{\log{t}}{\sigma_1 - \sigma_0} Y^{\sigma_1 - \sigma_A } \log{\frac{Y}{\sigma_A - \sigma_1}}.
\end{equation}
To compute \eqref{jerry}, we have divided the integral into two integrals , which are $\left|\xi \right| \leqslant 1 $ and $1 < \left| \xi  \right| \leqslant Y.$ 
Combining \eqref{Perron}, \eqref{residue}, \eqref{tom} and \eqref{jerry}, we obtain the second claim as well. 
\end{proof}
\par
Similar to Lemma \ref{approximatezeta}, we have the following lemma on Dirichlet $L$-functions. 
In fact, this lemma is similar to \cite[Lemma 2]{yang2023omega}.
\begin{lemma}
	\label{approximateL}
       Let $A$ be any positive real number, $q \geqslant 3$ be a prime number, $\sigma_{A} \operatorname{:} = 1-A/\log_2 q.$
	Let $Y \geqslant 3, -3q \leqslant t \leqslant 3q $ and $\frac{1}{2} \leqslant \sigma_0 < 1.$ Suppose that the rectangle $\{s:\sigma_0 <\operatorname{Re}\left(s\right) \leqslant 1,\lvert \operatorname{Im}\left(s\right) - t \rvert \leqslant Y+2 \}$ is
	free of zeros of $L \left(s, \chi \right),$ where $\chi$ is a non-principal character modulo $q$. Then for $\sigma_A \left(\leqslant 3\right)$ and any
	$\xi \in \left[t-Y, t+Y\right],$ we have 
	$$\left|\frac{L^{\prime}}{L} \left(\sigma_A + i \xi , \chi \right)\right| \ll \frac{\log{q}}{\sigma_A - \sigma_0}.$$
	\par 
	Further, for $\sigma_1 \in \left(\sigma_0, \sigma_A\right),$ we have
	$$-\frac{L ^{\prime}}{L}\left(\sigma_{A} + it , \chi \right)= \sum_{n \leqslant Y}\frac{\Lambda \left(n\right)\chi \left(n\right)}{n^{\sigma_{A} + it}}  + O\left(\frac{\log{q}}{\sigma_{1}-\sigma_{0}}Y^{\sigma_{1}-\sigma_{A}} \log {\frac{Y}{\sigma_{A}-\sigma_{1}}}\right).$$
\end{lemma}
\begin{proof}
	We write $\rho^{\prime} = \alpha^{\prime} +i \gamma^{\prime}, 0 < \alpha^{\prime} <1  $ as the non-trivial zeros of $L\left(s,\chi \right).$ By \cite[Lemma 11.4]{Koukoulopoulos2019}, we have 
	$$\left|\frac{L^{\prime}}{L}\left(\sigma_{A} + i\xi ,\chi  \right) \right| \leqslant \sum_{\left|\gamma^{\prime} -\xi \right| \leqslant 1} \frac{1}{\sigma_{A}-\alpha^{\prime} } + O\left(\log{q }\right) \ll \frac{\log{q}}{\sigma_{A}-\sigma_{0}}.$$
	The following steps are the same as in the proof of Lemma \ref{approximatezeta}, so we omitted them.
\end{proof}

\section{Proof of Theorem \ref{thm1}}
\label{proofthm1}
We still need the zero-density result of $\zeta\left(s\right).$ Let $N\left(\sigma, T\right)$ denote the number of zeros of $\zeta \left(s\right)$
inside the rectangle $\{s: \operatorname{Re}\left(s\right) \geqslant \sigma, 0 < \operatorname{Im}\left(s\right) \leqslant T \}.$
For $\frac{1}{2} \leqslant \sigma_0 \leqslant 1$ and $T \geqslant 2,$ Ingham \cite{Ingham1940} obtained a classical result 
$$N\left(\sigma_0, T\right) \ll T^{\frac{3\left(1-\sigma_0\right)}{2-\sigma_0}}\log ^ 5 T.$$
\par
Fix $\epsilon \in \left(0,\sigma_A - \frac{1}{2}\right).$ We set $\sigma_0 = \sigma_A - \epsilon, $ $\sigma_1 = \sigma_0 + 1/\log{Y},$ 
and $Y=\left(\log {T}\right)^{20/ \epsilon}.$ Combining with Ingham's zero-density result of $\zeta \left(s\right)$ we establish the following approximate formula
\begin{equation}
	\label{final1}
	- \frac{\zeta ^{\prime}}{\zeta} \left(\sigma_A + i t\right)= \sum_{n \leqslant Y} \frac{\Lambda \left(n\right)}{n^{\sigma_A + i t}} +O\left(\left(\log{T}\right)^{-18}\right),
	~ \forall t \in [T^{\beta}, T] \setminus \mathfrak{B}  \left(T\right),
\end{equation}where $\mathfrak{B} \left(T\right) $ is a ``bad'' set satisfying
\begin{equation}
	\label{meas bad set}
	\operatorname{meas} \left(\mathfrak{B} \left(T\right)\right) \ll T^{\frac{3\left(1-\sigma_A + \epsilon\right)}{2-\sigma_A + \epsilon}}\left(\log{T}\right)^{5+ \frac{20}{\epsilon}}.
\end{equation}
\par
Next, we take $X = \kappa \log{T} \log_2 {T}.$ Define $\Phi \left(t\right)= e^{- t^2/2}$ as in \cite{Bondarenko2017}, we note that for any $ \xi \in \mathbb{R},$ the Fourier transform of $\Phi $ is
$$\hat{\Phi }\left(\xi \right) \operatorname{:} = \int_{- \infty}^{\infty}\Phi \left(x\right) e^{-ix\xi } \mathrm{d}x = \sqrt{2 \pi }\Phi \left(\xi \right),$$
which is always positive. 
Let $r \left(n\right)$ be a completely multiplicative function with values at primes as 
\begin{equation}
	\label{rp}
	r\left(p\right) = 
	\begin{cases}
		1-X^{\sigma_A - 1}, ~ &\operatorname{if} p \leqslant X, \\
		0, ~ &\operatorname{if} p > X.
	\end{cases}
\end{equation}
\par
Define
$$R\left(t\right)  = \prod_{p} \frac{1}{1- r\left(p\right) p^{-it}} = \sum_{n \in \mathbb{N}} r \left(n\right) n ^{- i t} .$$
Notice that 
$$\log{\left| R \left(t\right) \right|} \leqslant \sum_{p \leqslant X} \log{| \frac{1}{1 - r\left(p\right)} |} = \sum_{p \leqslant X} \left(1 - \sigma_A\right) \log{X},$$
where the first inequality holds by the fact that $\left| p^{it} - r\left(p\right)\right| \geqslant \left| \left| p^{it} \right| - \left| r\left(p\right) \right|\right| = \left|1 - r\left(p\right) \right|.$
By the prime number theorem, we obtain that 
\begin{equation}
	\label{Rt}
	\left| R \left(t\right)\right| ^ 2 \leqslant T^{2A \kappa + o\left(1\right)}.
\end{equation}
\par
Define $M_2 \left(R,T\right)$, $M_1 \left(R,T\right)$, $I_2 \left(R,T\right)$ and $I_1 \left(R,T\right)$ as follows.
\begin{align*}
	M_2 \left(R,T\right)=&\int_{T^{\beta}}^{T} \operatorname{Re} \bigg(\sum_{n \leqslant Y} \frac{\Lambda \left(n\right)}{n^{\sigma_A + i t}}\bigg)\left|R\left(t\right) \right|^2 \Phi \left(\frac{t \log{T}}{T}\right) \mathrm{d}t,\\
	M_1 \left(R,T\right)=&\int_{T^{\beta}}^{T} \left|R\left(t\right) \right|^2 \Phi \left(\frac{t \log{T}}{T}\right) \mathrm{d}t, \\
	I_2 \left(R,T\right)=&\int_{-\infty}^{+\infty} \operatorname{Re} \bigg(\sum_{n \leqslant Y} \frac{\Lambda \left(n\right)}{n^{\sigma_A + i t}}\bigg)\left|R\left(t\right) \right|^2 \Phi \left(\frac{t \log{T}}{T}\right) \mathrm{d}t, \\
	I_1 \left(R,T\right)=&\int_{-\infty}^{+\infty} \left|R\left(t\right) \right|^2 \Phi \left(\frac{t \log{T}}{T}\right) \mathrm{d}t.
\end{align*}
Plainly, we have
\begin{equation}
	\label{final2}
	\max_{T^{\beta} \leqslant t \leqslant T} \operatorname{Re} \bigg(\sum_{n \leqslant Y} \frac{\Lambda \left(n\right)}{n^{\sigma_A + i t}}\bigg) \geqslant \frac{M_2 \left(R,T\right)}{M_1 \left(R,T\right)}.
\end{equation}
\par
Clearly, we have the following crude bound by the definition of $Y$
\begin{equation}
	\label{ReL}
	\bigg| \operatorname{Re} \bigg(\sum_{n \leqslant Y} \frac{\Lambda \left(n\right)}{n^{\sigma_A + i t}}\bigg) \bigg| \leqslant \sum_{n \leqslant Y} \frac{\Lambda \left(n\right)}{n^{\sigma_A}}
   \ll Y \ll T^{o \left(1\right)}.
\end{equation}
By \eqref{Rt}, \eqref{ReL} and the definition of the function $\Phi ,$ we immediately obtain two upper bounds on partial integrals of $I_2 \left(R,T\right)$
\begin{align*}    
	&\bigg| \int_{\left| t \right| \leqslant T^{\beta}} \operatorname{Re} \bigg(\sum_{n \leqslant Y} \frac{\Lambda \left(n\right)}{n^{\sigma_A + it}}\bigg) \left| R \left(t\right)\right|^2 \Phi \left(\frac{t \log{T}}{T}\right) \mathrm{d}t \bigg| \ll T^{\beta + 2A \kappa + o\left(1\right)} ,\\
	&\bigg| \int_{\left| t \right| \geqslant T} \operatorname{Re} \bigg(\sum_{n \leqslant Y} \frac{\Lambda \left(n\right)}{n^{\sigma_A + it}}\bigg) \left| R \left(t\right)\right|^2 \Phi \left(\frac{t \log{T}}{T}\right) \mathrm{d}t \bigg| \ll 1 .
\end{align*}
\par
As a result, we find the following approximate formula
\begin{equation}
	\label{M2I2}
   2 M_2 \left(R,T\right) = I_2 \left(R,T\right) + O \left(T^{\beta + 2 A \kappa + o \left(1\right)}\right).
\end{equation}
Similarly, 
\begin{equation}
	\label{M1I1}
	2 M_1 \left(R,T\right) = I_1 \left(R,T\right) + O \left(T^{\beta + 2 A \kappa + o \left(1\right)}\right).
\end{equation}
\par
To give a lower bound for $I_1 \left(R,T\right),$ we expand the product of $R\left(t\right).$ Since the Fourier transform $\hat{\Phi}$ is always positive, we can only consider the contributions from the diagonal terms 
and drop off-diagonal terms: 
$$	I_1 \left(R,T\right)  = \int_{-\infty}^{+ \infty} \sum_{n,m \in \mathbb{N}} r\left(n\right) r\left(m\right) \left(\frac{n}{m}\right)^{i t} \Phi\left(\frac{t\log{T}}{T}\right)\mathrm{d}t 
 \geqslant \sqrt{2 \pi} \frac{T}{\log{T}} \sum_{n \in \mathbb{N}} r^{2} \left(n\right). $$
By properties of the multiplicative function $r\left(n\right)$, we have $r\left(1\right) = 1.$ 
Dropping all terms except $n = 1$ in the summation, we immediately get the following crude lower bound
$$I_1 \left(R,T\right) \geqslant \sqrt{2 \pi} \frac{T}{\log{T}} \sum_{n \in \mathbb{N}} r^{2} \left(n\right) \geqslant \sqrt{2 \pi} \frac{T}{\log{T}} \gg T^{1 + o\left(1\right)}.$$
\par
To make the error terms small, we require $\kappa$ is a positive number satisfying
\begin{equation}
	\label{kappa1}
	\beta + 2 A \kappa < 1.
\end{equation}
By \eqref{M2I2}, \eqref{M1I1}, $I_1 \left(R,T\right)\gg T^{1+ o \left(1\right)}$ and the trivial relation that $M_1 \left(R,T\right) \leqslant I_1 \left(R,T\right),$ we obtain 
\begin{equation}
	\label{M2M1I2I1}
	\frac{M_2 \left(R,T\right)}{M_1 \left(R,T\right)} \geqslant \frac{I_2 \left(R,T\right)}{I_1 \left(R,T\right)} + O \left(T^{\beta + 2 A \kappa - 1 + o \left(1\right)}\right).
\end{equation}
\par
Now we get
\begin{equation}
	\label{final4}
	\max_{T^{\beta} \leqslant t \leqslant T} \operatorname{Re} \bigg(\sum_{n \leqslant Y} \frac{\Lambda \left(n\right)}{n^{\sigma_A + i t}}\bigg) \geqslant \frac{I_2 \left(R,T\right)}{I_1 \left(R,T\right)} + O \left(T^{\beta + 2 A \kappa - 1 + o \left(1\right)}\right)
\end{equation}
by \eqref{final2} and \eqref{M2M1I2I1}.
\par
We still want to obtain the lower bound for ${I_2 \left(R,T\right)},$ so we compute $I_2 \left(R,T\right)$ directly
\begin{align*}
	I_2 \left(R,T\right) &=\int_{-\infty}^{+\infty} \operatorname{Re} \bigg(\sum_{n \leqslant Y} \frac{\Lambda \left(n\right)}{n^{\sigma_A + i t}}\bigg)\left|R\left(t\right) \right|^2 \Phi \left(\frac{t \log{T}}{T}\right) \mathrm{d}t \\
                         &= \sum_{p^{v} \leqslant Y} \frac{\log{p}}{p^{v \sigma_A}} \operatorname{Re} \bigg(\int_{- \infty}^{+ \infty} \sum_{n,m \in \mathbb{N}} r\left(n\right) r\left(m\right)\left(\frac{n}{p^v m}\right)^{it} \Phi \left(\frac{t \log{T}}{T}\right) \mathrm{d}t\bigg) \\
						 & \geqslant \sum_{p \leqslant X} \frac{\log{p}}{p^{\sigma_A}} \int_{- \infty}^{+ \infty} \sum_{\substack{k,m \in \mathbb{N} \\ n = pk }} r\left(n\right) r\left(m\right)\left(\frac{n}{p m}\right)^{it} \Phi \left(\frac{t \log{T}}{T}\right) \mathrm{d}t \\
						 &= \bigg(\sum_{p \leqslant X} \frac{\log{p}}{p^{\sigma_A}} r\left(p\right)\bigg) \int_{-\infty}^{+ \infty} \sum_{k,m \in \mathbb{N}} r\left(k\right) r\left(m\right)\left(\frac{k}{ m}\right)^{it} \Phi \left(\frac{t \log{T}}{T}\right) \mathrm{d}t \\
						 &= \bigg(\sum_{p \leqslant X} \frac{\log{p}}{p^{\sigma_A}} r\left(p\right)\bigg) I_1 \left(R,T\right)
\end{align*}
by the definitions of $\Lambda \left(n\right)$ and $I_1 \left(R,T\right).$ Thus, we have
\begin{equation}
	\label{final5}
	\frac{I_2 \left(R,T\right)}{I_1 \left(R,T\right)} \geqslant \sum_{p \leqslant X} \frac{\log{p}}{p^{\sigma_A}} r\left(p\right).
\end{equation}
\par
Now we compute the summation on the right-hand side of the above inequality by the prime number theorem and partial summation as $T \to \infty$
\begin{align*}
	\sum_{p \leqslant X} \frac{\log{p}}{p^{\sigma_A}} &=  \vartheta \left(X\right) X^{- \sigma_A} - \int_{2}^{X} \vartheta \left(t\right) \mathrm{d}\left(t^{-\sigma_A}\right) +O\left(1\right)  \\
	                                                  &=  X^{1-\sigma_A} + \sigma_A \int_{2}^{X} t^{-\sigma_A} \mathrm{d}t + O\left(\frac{X^{1-\sigma_A}}{\log{X}}\right) +O\left(\sigma_A \int_{2}^{X} \frac{t^{-\sigma_A}}{\log{t}} \mathrm{d}t\right) \\
													  &= \frac{1}{1 -\sigma_A} X^{1- \sigma_A} + O\left(\frac{X^{1-\sigma_A}}{\log{X}}\right) + O\left( \frac{\sigma_A}{1 - \sigma_A} \frac{X^{1-\sigma_A}}{\log{X}}\right) \\
													  &= \frac{1}{1 -\sigma_A} X^{1- \sigma_A} + O\bigg( \frac{1}{1 -\sigma_A} \frac{X^{1-\sigma_A}}{\log{X}}\bigg).
\end{align*}
To optimize the error terms, we use integration by parts. 
Furthermore, when $p \leqslant X,$  the value of $r\left(p\right)$ is independent of $p.$ So we have the following approximate formula
$$\sum_{p \leqslant X} \frac{\log{p}}{p^{\sigma_A}} r\left(p\right) = \left(\frac{1}{1 -\sigma_A} X^{1- \sigma_A} + O\left( \frac{1}{1 -\sigma_A} \frac{X^{1-\sigma_A}}{\log{X}}\right)\right)\left(1- X^{\sigma_A -1 } \right). $$
Recall that $X = \kappa \log{T} \log_2 T,$ for all sufficiently large $T,$ we obtain
$$\sum_{p \leqslant X} \frac{\log{p}}{p^{\sigma_A}} r\left(p\right) = \frac{e^A-1}{A} \log_2 T + O\left(\frac{\log_2 T}{\log_3 T }\right)$$
from the above asymptotic formula by directly calculation. By \eqref{final5}, we have that
\begin{equation}
	\label{final6}
	\frac{I_2 \left(R,T\right)}{I_1 \left(R,T\right)} \geqslant \frac{e^A-1}{A} \log_2 T + O\left(\frac{\log_2 T}{\log_3 T }\right).
\end{equation}
\par
By \eqref{meas bad set}, \eqref{Rt} and \eqref{ReL}, for the integrand of $M_2 \left(R,T\right),$ the integration over $\mathcal{B} \left(T\right)$ is at most 
$$\int_{\mathcal{B} \left(T\right)} \operatorname{Re} \bigg(\sum_{n \leqslant Y} \frac{\Lambda \left(n\right)}{n^{\sigma_A + i t}}\bigg)\left|R\left(t\right) \right|^2 \Phi \left(\frac{t \log{T}}{T}\right) \mathrm{d}t \ll T ^ {2 A \kappa + \frac{3\left(1-\sigma_A + \epsilon\right)}{2-\sigma_A + \epsilon}+ o\left(1\right)}.$$
If we require 
\begin{equation}
	\label{kappa2}
	2 A \kappa + \frac{3\left(1-\sigma_A + \epsilon\right)}{2-\sigma_A + \epsilon} < 1,
\end{equation}
this contribution will be negligible.
\par
Let $\kappa $ satisfies \eqref{kappa1} and \eqref{kappa2}, combining \eqref{final1}, \eqref{final4} and \eqref{final6}, we have 
$$\max_{T^{\beta} \leqslant t \leqslant T} - \operatorname{Re} \frac{\zeta ^{\prime} }{\zeta }\left( \sigma_{A} + it\right) \geqslant \frac{e^{A} -1}{A} \log_2 T + O\left(\frac{\log_2 {T}}{\log_3 {T}}\right) .$$
Now the proof of Theorem \ref{thm1} is completed.

\begin{remark}
	\label{remark1}
	In view of the definition of $\sigma_A,$ we find that $\sigma_A$ is close to the 1-line very much. 
	Specifically, $\left|\sigma_A -1  \right| \ll 1 / \log_2 T$ as $T \to \infty.$
	So we can use other more precise zero-density results. For instance, by {\rm \cite[Theorem 11.4]{ivic2012riemann}}
	$$N\left(\sigma, T\right) \ll T^{\frac{6\left(1-\sigma\right)}{5\sigma -1 }}, ~ \frac{13}{17} \leqslant \sigma \leqslant 1$$
    and 
	$$N\left(\sigma, T\right) \ll T^{2-2\sigma}, ~ \frac{11}{14} \leqslant \sigma \leqslant 1 .$$
    Using the above results, we can optimize the upper bound for the measure of $\mathfrak{B} \left(T\right) $, but there is no effect on our Theorem \ref{thm1}.
\end{remark}

\section{Proof of Theorem \ref{thm2}}
\label{proofthm2}
In fact, the proof is a slightly modification of the proof of Theorem \ref{thm1}. Set $Y= \left(\log{T}\right)^{20/\epsilon}$ and $X=\kappa \log {T} \log_2 T$ as in Section \ref{proofthm1}. 
Assuming RH, we can neglect the contribution from $\mathcal{B} \left(T\right),$ and immediately establish the following approximate formula from \eqref{final1}
\begin{equation}
	\label{RH approximate}
	- \frac{\zeta ^{\prime}}{\zeta} \left(\sigma_A + i t\right)= \sum_{n \leqslant Y} \frac{\Lambda \left(n\right)}{n^{\sigma_A + i t}} +O\left(\left(\log{T}\right)^{-18}\right),
	~ \forall t \in [T^{\beta}, T].
\end{equation}
Set 
$$\kappa = \frac{1- \beta}{2A} \operatorname{exp} \left(-x + \left(\log {T}\right)^{-E} + f\left(T\right)\right),$$
where $f\left(T\right) = \operatorname{exp}\left(- \left(\log_2 T \right)\right),$ and $E$ is slightly smaller than 18. Obviously, for all sufficiently large $T,$ 
we obtain $\beta - 1 + 2A \kappa < \beta - 1 + \left(1- \beta \right)\operatorname{exp}\left(-\frac{1}{2}x\right) < 0 .$
\par
Define $J_x$ and $\widetilde{J_x}$ as follows
\begin{align*}
	J_x &= \frac{e^A- 1}{A} \log_3 T -x + \left(\log T\right)^{-E} + \frac{1}{2} f\left(T\right), \\
	\widetilde{J_x} & = \frac{e^A- 1}{A} \log_3 T -x + \left(\log T\right)^{-E}. 
\end{align*}
Then by \eqref{M2M1I2I1}, \eqref{final5} and \eqref{final6}, we have 
\begin{equation}
	\label{M2M1J}
	\frac{M_2 \left(R,T\right)}{M_1 \left(R,T\right)} \geqslant \frac{\log_2 T}{\log_3 T} \cdot J_x
\end{equation}
for all sufficiently large $T.$
\par
Define $V_x, W_x$ and $Z_x$ as follows 
\begin{align*}
	V_x & = \bigg\{t \in [T^{\beta} , T] \operatorname{:} \operatorname{Re} \bigg(\sum_{n \leqslant Y} \frac{\Lambda \left(n\right)}{n^{\sigma_A + i t}}\bigg) \leqslant \frac{\log_2 T}{\log_3 T} \cdot \widetilde{J_x}\bigg\} , \\
	W_x & = \bigg\{t \in [T^{\beta} , T] \operatorname{:} \operatorname{Re} \bigg(\sum_{n \leqslant Y} \frac{\Lambda \left(n\right)}{n^{\sigma_A + i t}}\bigg) > \frac{\log_2 T}{\log_3 T} \cdot  \widetilde{J_x}\bigg\} , \\
	Z_x & = \bigg\{t \in [T^{\beta} , T] \operatorname{:} - \operatorname{Re} \frac{\zeta ^{\prime}}{\zeta} \left(\sigma_A + i t\right) > \frac{\log_2 T}{\log_3 T} \cdot \left(\widetilde{J_x} - \left(\log{T}\right)^{-E}\right)\bigg\} .
\end{align*}
By the definitions, we have $V_x \cap W_x = \varnothing$ and $V_x \cup W_x = [T^{\beta}, T].$ Recall the definition of $M_2 \left(R, T\right),$ we have
\begin{align*}
	M_2\left(R,T\right) & = \int_{T^{\beta}}^{T} \operatorname{Re} \bigg(\sum_{n \leqslant Y} \frac{\Lambda \left(n\right)}{n^{\sigma_A + i t}}\bigg)\left|R\left(t\right) \right|^2 \Phi \left(\frac{t \log{T}}{T}\right) \mathrm{d}t \\
	                    & = \left(\int_{V_x} + \int_{W_x}\right) \operatorname{Re} \bigg(\sum_{n \leqslant Y} \frac{\Lambda \left(n\right)}{n^{\sigma_A + i t}}\bigg)\left|R\left(t\right) \right|^2 \Phi \left(\frac{t \log{T}}{T}\right) \mathrm{d}t \\
	                    & \leqslant \frac{\log_2 T}{\log_3 T} \cdot \widetilde{J_x} \cdot M_1\left(R,T\right) + \int_{W_x} \operatorname{Re} \bigg(\sum_{n \leqslant Y} \frac{\Lambda \left(n\right)}{n^{\sigma_A + i t}}\bigg)\left|R\left(t\right) \right|^2 \Phi \left(\frac{t \log{T}}{T}\right) \mathrm{d}t.
\end{align*}
Combining the above computation with \eqref{M2M1J}, we have
\begin{equation}
	\label{1/2f}
	\frac{\log_2 T}{2 \log_3 T} f\left(T\right) \cdot M_1\left(R,T \right) \leqslant \int_{W_x} \operatorname{Re} \bigg(\sum_{n \leqslant Y} \frac{\Lambda \left(n\right)}{n^{\sigma_A + i t}}\bigg)\left|R\left(t\right) \right|^2 \Phi \left(\frac{t \log{T}}{T}\right) \mathrm{d}t.
\end{equation}
\par 
We have already obtained $\left| R \left(t\right)\right| ^ 2 \leqslant T^{2A \kappa + o\left(1\right)}$ from \eqref{Rt}. Meanwhile, by the prime number theorem, we have
$$	\bigg| \operatorname{Re} \bigg(\sum_{n \leqslant Y} \frac{\Lambda \left(n\right)}{n^{\sigma_A + i t}}\bigg) \bigg| \leqslant \sum_{n \leqslant Y} \frac{\Lambda \left(n\right)}{n^{\sigma_A}} \ll \log {Y} \ll \log_2 T.$$
By above estimates and \eqref{RH approximate}, we get
\begin{equation}
	\label{measWx}
	\int_{W_x} \operatorname{Re} \bigg(\sum_{n \leqslant Y} \frac{\Lambda \left(n\right)}{n^{\sigma_A + i t}}\bigg)\left|R\left(t\right) \right|^2 \Phi \left(\frac{t \log{T}}{T}\right) \mathrm{d}t 
	\leqslant \left(\log_2 {T}\right) T^{2A\kappa + o\left(1\right)} \operatorname{meas}\left(W_x\right).
\end{equation}
Combining \eqref{M1I1} and $I_1\left(R,T\right) \geqslant T^{1+o\left(1\right)},$ we obtain $M_1 \left(R,T\right) \geqslant T^{1+o \left(1\right)}$
because $\kappa$ satisfies $\beta - 1 + 2A \kappa < 0.$ Furthermore, by \eqref{1/2f} and \eqref{measWx}, we have
$$\operatorname{meas}\left(W_x\right) \geqslant \frac{\frac{1}{2}f\left(T\right)\log_2 {T} \cdot M_1\left(R,T\right)}{\left(\log_2 T \log_3 T\right) T^{2A\kappa + o\left(1\right)}}
\geqslant T^{\left(1+o\left(1\right)\right)\left(1-2A\kappa\right)}.$$
\par 
By the definition of $\kappa,$ we have $\operatorname{meas}\left(W_x\right) \geqslant T^{1-\left(1-\beta\right)e^{-x}+o\left(1\right)}.$ As $W_x$ is a subset of $Z_x,$ we get $\operatorname{meas}\left(Z_x\right) \geqslant \operatorname{meas}\left(W_x\right).$
So we show that Theorem \ref{thm2} holds. 

\section{Proof of Theorem \ref{thm3}}
\label{proofthm3}
We denote by $G_q$ the set of Dirichlet characters $\chi  \left(\operatorname{mod}~q\right).$ Long before, Gronwall and Titchmarsh (for instance, see \cite[Theorem 11.3]{montgomery2007multiplicative}) had given a classical result, which is about the zero-free region on Dirichlet $L$-functions. 
They proved that there exists a constant $ C > 0,$ such that for any $\chi \in G_q,$ the
region 
$$\bigg\{s=\sigma + it \operatorname{:} \sigma > 1 - \frac{C}{\log{\left(q \left( \left|t \right| + 2\right)\right)}}\bigg\}$$
is free of zeros of $L \left(s, \chi \right),$ unless $\chi = \chi_e$ is a quadratic character. In this case, $L \left(s, \chi \right)$ has at most one zero in this region.
We define $G_q ^{\ast} = G_q \setminus \{\chi _0, \chi _e\}.$
\par
Similar to Section \ref{proofthm1}, we need the zero-density result of $L \left(s, \chi \right).$ 
Let $N\left(\sigma, T, \chi \right)$ denote the number of zeros of $L \left(s,\chi \right)$ inside the rectangle $\{s: \operatorname{Re}\left(s\right) \geqslant \sigma, \left|\operatorname{Im} \left(s\right) \right| \leqslant T\}.$
For $\frac{1}{2} \leqslant \sigma_0 \leqslant 1$ and $T \geqslant 2,$ Montgomery \cite{Montgomery1971} obtained a similar zero-density result to Ingham's result on Dirichlet $L$-functions as follows
$$\sum_{\chi  \in G_q}N\left(\sigma_0, T, \chi \right) \ll \left(qT\right)^{\frac{3\left(1-\sigma_0\right)}{2-\sigma_0}}\log ^ {14} \left(qT\right).$$
\par
Fix $\epsilon \in \left(0,\sigma_A - \frac{1}{2}\right).$ We set $t = 0,~ \sigma_0 = \sigma_A - \epsilon / 2,$ $\sigma_1 = \sigma_0 + 1/\log{Y}$  
and $Y=\left(\log {q}\right)^{20/\epsilon}$ in Lemma \ref{approximateL}. Taking $T=Y+2$ and combining Montgomery's zero-density result, we establish the following approximate formula
\begin{equation}
	\label{finalq1}
	- \frac{L^{\prime}}{L} \left(\sigma_A, \chi\right)= \sum_{n \leqslant Y} \frac{\Lambda \left(n\right)\chi \left(n \right)}{n^{\sigma_A}} +O\left(\left(\log{q}\right)^{-18}\right),
	~ \forall \chi  \in G_q ^{\ast} \setminus B_q,
\end{equation}
where $B_q$ is a set of ``bad'' characters with cardinality satisfying
\begin{equation}
	\label{meas bad setq}
	\# B_q \ll q^{\frac{3\left(1-\sigma_A + \frac{1}{2}\epsilon\right)}{2-\sigma_A +\frac{1}{2}\epsilon}} \left(\log{q}\right)^{O\left(1\right)}.
\end{equation}
\par
Let $X = \kappa \log{q} \log_2{q}.$ And we let $r\left(n\right)$ be the same function as in Section \ref{proofthm1}. Meanwhile, we define $R\left(\chi\right) = \sum_{n \in \mathbb{N}} r\left(n\right) \chi \left(n\right).$
By the prime number theorem, we can obtain the following bound similar to Section \ref{proofthm1}
\begin{equation}
	\label{Rchi}
	\left| R\left(\chi\right) \right| ^2 \leqslant q^{2 A \kappa + o\left(1\right)}.
\end{equation}
\par
Define $S_2 \left(R,q\right)$, $S_1 \left(R,q\right)$, $S_2 ^{\ast} \left(R,q\right)$ and $S_1 ^{\ast} \left(R,q\right)$ as follows
\begin{align*}
	S_2 \left(R,q\right) &= \sum_{\chi \in G_q} \operatorname{Re} \bigg(\sum_{n \leqslant Y}\frac{\Lambda \left(n\right)\chi \left(n\right)}{n^{\sigma_A}}\bigg) \left|R\left(\chi \right) \right| ^2, \\
    S_1 \left(R,q\right) &= \sum_{\chi \in G_q} \left|R\left(\chi \right) \right| ^2, \\
	S_2 ^{\ast} \left(R,q\right) &= \sum_{\chi \in G_q^{\ast } \setminus B_q} \operatorname{Re} \bigg(\sum_{n \leqslant Y}\frac{\Lambda \left(n\right) \chi \left(n\right)}{n^{\sigma_A}} \bigg) \left|R\left(\chi \right) \right| ^2, \\
	S_1 ^{\ast} \left(R,q\right) &= \sum_{\chi \in G_q^{\ast } \setminus B_q} \left|R\left(\chi \right) \right| ^2.
\end{align*}
Plainly, we have
\begin{equation}
	\label{finalq2}
	\max_{\chi \in G_q^{\ast } \setminus B_q} \operatorname{Re} \bigg(\sum_{n \leqslant Y} \frac{\Lambda \left(n\right)\chi \left(n\right)}{n^{\sigma_A}}\bigg) \geqslant \frac{S_2 ^{\ast} \left(R,q\right)}{S_1 ^{\ast} \left(R,q\right)}.
\end{equation}
\par
Meanwhile, we have
\begin{equation}
	\label{ReLchi}
	\operatorname{Re} \bigg(\sum_{n \leqslant Y}\frac{\Lambda \left(n\right)\chi \left(n\right)}{n^{\sigma_A}}\bigg) \ll q^{o\left(1\right)}.
\end{equation}
By \eqref{meas bad setq}, \eqref{Rchi} and \eqref{ReLchi}, we obtain the following approximate formulas
\begin{align*}
	S_2 \left(R,q\right) &= S_2^{\ast} \left(R,q\right) + O\left(q^{2 A \kappa + \frac{3\left(1- \sigma_A + \epsilon\right)}{2- \sigma_A + \epsilon} + o\left(1\right)}\right), \\
	S_1 \left(R,q\right) &= S_1^{\ast} \left(R,q\right) + O\left(q^{2 A \kappa + \frac{3\left(1- \sigma_A + \epsilon\right)}{2- \sigma_A + \epsilon} + o\left(1\right)}\right).
\end{align*}
\par
To give a lower bound for $S_1 \left(R,q\right),$ we make use of the orthogonality of Dirichlet characters and the fact that $r\left(1\right)= 1 :$
\begin{align*}
	S_1 \left(R,q\right) &= \sum_{\chi \in G_q} \left|R\left(\chi \right) \right| ^2 \\
	                     &= \sum_{n,k \in \mathbb{N}} r\left(n\right) r\left(k\right) \sum_{\chi \in G_q} \chi \left(n\right) \bar \chi \left(k\right) \\
						 &= \phi \left(q\right)\sum_{\substack{n,k \in \mathbb{N} \\ n \equiv k \left(\operatorname{mod}~q\right) \\ \operatorname{gcd}\left(k,q\right)=1}}  r\left(n\right) r \left(k\right) \\
						 & \geqslant \left(q - 1\right) \sum_{n \in \mathbb{N}} r^2 \left(n\right) \\
						 & \gg q^{1 + o\left(1\right)}.
\end{align*}
\par
Combining the above approximate formulas and the trivial relation $S_1^{\ast} \left(R,q\right) \leqslant S_1 \left(R,q\right),$ we get
\begin{equation}
	\label{S2S1starS2S1}
	\frac{S_2^{\ast}  \left(R,q\right)}{S_1^{\ast} \left(R,q\right)} \geqslant \frac{S_2\left(R,q\right)}{S_1 \left(R,q\right)} + O\left(q^{2 A \kappa + \frac{3\left(1- \sigma_A + \epsilon\right)}{2- \sigma_A + \epsilon} -1 + o\left(1\right)}\right).
\end{equation}
To make the error terms small, we require
\begin{equation}
	\label{kappa3}
	2 A \kappa + \frac{3\left(1- \sigma_A + \epsilon\right)}{2- \sigma_A + \epsilon} -1 < 0.
\end{equation}
\par
By \eqref{finalq2} and \eqref{S2S1starS2S1}, we get
\begin{equation}
	\label{finalq4}
    \max_{\chi \in G_q ^{\ast} \setminus B_q} \operatorname{Re} \bigg(\sum_{n \leqslant Y}\frac{\Lambda \left(n\right)\chi \left(n\right)}{n^{\sigma_A}}\bigg) \geqslant \frac{S_2 \left(R,q\right)}{S_1 \left(R,q\right)} + O\left(q^{2 A \kappa + \frac{3\left(1- \sigma_A + \epsilon\right)}{2- \sigma_A + \epsilon} -1 + o\left(1\right)}\right).
\end{equation}
We still want to obtain the lower bound for ${S_2 \left(R,q\right)}.$ We compute $S_2 \left(R,q\right)$ directly
\begin{align*}
	S_2 \left(R,q\right) &= \sum_{\chi \in G_q} \operatorname{Re} \bigg(\sum_{n \leqslant Y}\frac{\Lambda \left(n\right)\chi \left(n\right)}{n^{\sigma_A}}\bigg) \left|R\left(\chi \right) \right| ^2 \\
	                     &= \sum_{p^v \leqslant Y} \frac{\log{p}}{p^{v \sigma_A}} \operatorname{Re} \bigg(\sum_{n,m \in \mathbb{N}} r\left(n\right) r\left(m\right) \sum_{\chi \in G_q} \chi \left(p^v n\right) \bar\chi \left(m\right)\bigg) \\
	                     & \geqslant \sum_{p \leqslant Y} \frac{\log{p}}{p^{ \sigma_A}} \phi \left(q\right) \sum_{\substack{n,m \in \mathbb{N} \\ p n \equiv m \left(\operatorname{mod}~q\right) \\ \operatorname{gcd}\left(m,q\right)=1}}  r\left(n\right) r \left(m\right) \\
						 & \geqslant \bigg(\sum_{p \leqslant X} \frac{\log{p}}{p^{ \sigma_A}} r\left(p\right)\bigg) \phi \left(q\right) \sum_{\substack{n,k \in \mathbb{N} \\ p n \equiv pk \left(\operatorname{mod}~q\right) \\ \operatorname{gcd}\left(pk,q\right)=1}}  r\left(n\right) r \left(k\right).
\end{align*}
Since $q$ is a prime, by \cite{yang2024extreme} we have
$$\phi \left(q\right) \sum_{\substack{n,k \in \mathbb{N} \\ pn \equiv pk \left(\operatorname{mod}~q\right) \\ \operatorname{gcd}\left(pk,q\right)=1}}  r\left(n\right) r \left(k\right) = 
\phi \left(q\right) \sum_{\substack{n,k \in \mathbb{N} \\ n \equiv k \left(\operatorname{mod}~q\right) \\ \operatorname{gcd}\left(k,q\right)=1}}  r\left(n\right) r \left(k\right) = S_1 \left(R,q\right). $$
So we obtain
\begin{equation}
	\label{finalq5}
	\frac{S_2 \left(R,q\right)}{S_1 \left(R,q\right)} \geqslant \sum_{p \leqslant Y} \frac{\log{p}}{p^{ \sigma_A}} r\left(p\right).
\end{equation}
Recall $X = \kappa \log{q} \log_2{q},$ as $q \to \infty,$ we immediately get 
\begin{equation}
	\label{finalq6}
	\frac{S_2 \left(R,q\right)}{S_1 \left(R,q\right)} \geqslant \frac{e^A-1}{A} \log_2 q + O\left(\frac{\log_2 q}{\log_3 q}\right).
\end{equation}
\par
Let $\kappa$ satisfies \eqref{kappa3}, by \eqref{finalq1}, \eqref{finalq4} and \eqref{finalq6} we have
$$\max_{\substack{\chi \neq \chi_{0} \\ \chi \left(\operatorname{mod}~ q\right)}} - \operatorname{Re} \frac{L^{\prime}}{L}\left(\sigma_{A},\chi \right) \geqslant \frac{e^{A} -1}{A} \log_2 q + O\left(\frac{\log_2 {q}}{\log_3 {q}}\right).$$
Now the proof of Theorem \ref{thm3} is completed.

\section{Proof of Theorem \ref{thm4}}
\label{proofthm4}
Again, the proof is a slightly modification of the proof of Theorem \ref{thm3} and is similar as the proof of Theorem \ref{thm2}.
Set $Y= \left(\log{q}\right)^{20/\epsilon}$ and $X=\kappa \log {q} \log_2 q$ as in Section \ref{proofthm3}.
Assuming GRH, we can neglect the contribution from $B_q,$ so we immediately establish the following approximate formula from \eqref{finalq1}
\begin{equation}
	\label{GRH approximate}
	- \frac{L^{\prime}}{L} \left(\sigma_A, \chi\right)= \sum_{n \leqslant Y} \frac{\Lambda \left(n\right)\chi \left(n \right)}{n^{\sigma_A}} +O\left(\left(\log{q}\right)^{-18}\right),
	~ \forall \chi  \in G_q ^{\ast}.
\end{equation}
Meanwhile, we can rewrite the definitions of $S_2 ^{\ast} \left(R,q\right)$ and $S_1 ^{\ast} \left(R,q\right)$ as follows
\begin{align*}
	S_2 ^{\ast} \left(R,q\right) &= \sum_{\chi \in G_q^{\ast }} \operatorname{Re} \bigg(\sum_{n \leqslant Y}\frac{\Lambda \left(n\right)\chi \left(n\right)}{n^{\sigma_A}}\bigg) \left|R\left(\chi \right) \right| ^2, \\
	S_1 ^{\ast} \left(R,q\right) &= \sum_{\chi \in G_q^{\ast }} \left|R\left(\chi \right) \right| ^2.
\end{align*}
\par
We derive $\left| R\left(\chi\right) \right| ^2 \leqslant q^{2 A \kappa + o\left(1\right)}$ from \eqref{Rchi}, and  
\begin{equation}
	\label{Rchi1}
	\operatorname{Re} \bigg(\sum_{n \leqslant Y} \frac{\Lambda \left(n\right) \chi \left(n\right)}{n^{\sigma_A}}\bigg) \leqslant \sum_{n \leqslant Y} \frac{\Lambda \left(n\right)}{n^{\sigma_A}} \ll \log {Y} \ll \log_2 q
\end{equation}
by the prime number theorem. 
Naturally, we find that $S_1 \left(R,q\right) = S_1 ^{\ast} \left(R,q\right) + O\left(q^{2A \kappa + o\left(1\right)}\right)$ and $S_1 \left(R,q\right) \geqslant q ^{1+o \left(1\right)}.$
\par
We set $$\kappa = \frac{1}{2A} \operatorname{exp} \left(-x + \left(\log {q}\right)^{-E} + f\left(q\right)\right),$$
where $f\left(q\right) = \operatorname{exp}\left(- \left(\log_2 q \right)\right),$ and $E$ is slightly smaller than 18. Obviously, as $q \to \infty,$ 
we have $-1 + 2A \kappa < -1 + \operatorname{exp}\left(-\frac{1}{2}x\right) < 0 .$
\par
Define $J_x$ and $\widetilde{J_x}$ as follows
\begin{align*}
	J_x &= \frac{e^A- 1}{A} \log_3 q -x + \left(\log q\right)^{-E} + \frac{1}{2} f\left(q\right), \\
	\widetilde{J_x} & = \frac{e^A- 1}{A} \log_3 q -x + \left(\log q\right)^{-E}. 
\end{align*}
Then by the new definitions of $S_2^{\ast}  \left(R,q\right), S_1^{\ast}  \left(R,q\right),$  \eqref{finalq5} and \eqref{finalq6}, we obtain that 
\begin{equation}
	\label{S2S1J}
	\frac{S_2^{\ast}  \left(R,q\right)}{S_1^{\ast}  \left(R,q\right)} \geqslant \frac{\log_2 q}{\log_3 q} \cdot J_x
\end{equation}
as $q \to \infty.$
\par
Define $V_x, W_x$ and $Z_x$ as follows 
\begin{align*}
	V_x & = \bigg\{\chi \in G_q^{\ast } \operatorname{:} \operatorname{Re} \bigg(\sum_{n \leqslant Y} \frac{\Lambda \left(n\right) \chi \left(n\right)}{n^{\sigma_A}}\bigg) \leqslant \frac{\log_2 q}{\log_3 q} \cdot \widetilde{J_x}\bigg\} , \\
	W_x & = \bigg\{\chi \in G_q^{\ast } \operatorname{:} \operatorname{Re} \bigg(\sum_{n \leqslant Y} \frac{\Lambda \left(n\right) \chi \left(n\right)}{n^{\sigma_A}}\bigg) > \frac{\log_2 q}{\log_3 q} \cdot  \widetilde{J_x}\bigg\} , \\
	Z_x & = \bigg\{\chi \in G_q^{\ast } \operatorname{:} - \operatorname{Re} \frac{L^{\prime}}{L} \left(\sigma_A, \chi\right) > \frac{\log_2 q}{\log_3 q} \cdot \left(\widetilde{J_x} - \left(\log{T}\right)^{-E}\right)\bigg\} .
\end{align*}
Then we have  
\begin{align*}
	S_2^{\ast}  \left(R,q\right) & = \bigg(\sum_{\chi \in V_x} + \sum_{\chi \in W_x}\bigg) \operatorname{Re} \bigg(\sum_{n \leqslant Y}\frac{\Lambda \left(n\right)\chi \left(n\right)}{n^{\sigma_A}}\bigg) \left|R\left(\chi \right) \right| ^2\\
	                             & \leqslant \frac{\log_2 q}{\log_3 q} \cdot \widetilde{J_x} \cdot S_1^{\ast}  \left(R,q\right)  +  \sum_{\chi \in W_x} \operatorname{Re} \bigg(\sum_{n \leqslant Y}\frac{\Lambda \left(n\right)\chi \left(n\right)}{n^{\sigma_A}}\bigg) \left|R\left(\chi \right) \right| ^2.
\end{align*}
Combining the above computation with \eqref{S2S1J}, we have
\begin{equation}
	\label{1/2fq}
	\frac{\log_2 q}{2 \log_3 q} f\left(q\right) \cdot S_1^{\ast}  \left(R,q\right)  \leqslant \sum_{\chi \in W_x} \operatorname{Re} \bigg(\sum_{n \leqslant Y}\frac{\Lambda \left(n\right)\chi \left(n\right)}{n^{\sigma_A}}\bigg) \left|R\left(\chi \right) \right| ^2.
\end{equation}
By \eqref{Rchi}, \eqref{GRH approximate} and \eqref{Rchi1}, we obtain
\begin{equation}
	\label{jingWx}
	\sum_{\chi \in W_x} \operatorname{Re} \bigg(\sum_{n \leqslant Y}\frac{\Lambda \left(n\right)\chi \left(n\right)}{n^{\sigma_A}}\bigg) \left|R\left(\chi \right) \right| ^2
	\leqslant \left(\log_2 {q}\right) q^{2A\kappa + o\left(1\right)} \#\left(W_x\right).
\end{equation}
\par
Combining $S_1 \left(R,q\right) = S_1 ^{\ast} \left(R,q\right) + O\left(q^{2A \kappa + o\left(1\right)}\right)$ and $S_1 \left(R,q\right) \geqslant q ^{1+o \left(1\right)},$ we obtain $S_1 ^{\ast} \left(R,q\right) \geqslant q^{1+o \left(1\right)}$
because $\kappa$ satisfies $- 1 + 2A \kappa < 0.$ Furthermore, by \eqref{1/2fq} and \eqref{jingWx}, we get 
$$\#W_x \geqslant \frac{\frac{1}{2}f\left(q\right)\log_2 {q} \cdot S_1 ^{\ast} \left(R,q\right)}{\left(\log_2 q \log_3 q\right) q^{2A\kappa + o\left(1\right)}}
\geqslant q^{\left(1+o\left(1\right)\right)\left(1-2A\kappa\right)}.$$
\par 
By the definition of $\kappa,$ we know $\#W_x \geqslant q^{1-\left(1-\beta\right)e^{-x}+o\left(1\right)}.$ As $W_x$ is a subset of $Z_x,$ we get $\#Z_x \geqslant \#W_x.$
So we show that Theorem \ref{thm4} holds.

\bibliographystyle{siam}
\bibliography{2asset}
\end{document}